\documentclass{amsart}

\usepackage{amssymb}
\usepackage{amsmath}
\usepackage{amscd}
\usepackage{amsthm}
\usepackage{latexsym}

\newtheorem{thm}{Theorem}[section]
\newtheorem{prop}[thm]{Proposition}
\newtheorem{lem}[thm]{Lemma}
\newtheorem{cor}[thm]{Corollary}
\newtheorem{rem}[thm]{Remark}

\newcommand{\lorarr}{\longrightarrow}

\newcommand{\Rarr}{\Rightarrow}

\newcommand{\mcF}{\mathcal{F}}

\newcommand{\mcO}{\mathcal{O}}

\newcommand{\mcU}{\mathcal{U}}

\newcommand{\mcX}{\mathcal{X}}

\DeclareMathOperator{\Cl}{Cl}

\DeclareMathOperator{\Syl}{Syl}

\DeclareMathOperator{\res}{res}

\DeclareMathOperator{\Tr}{Tr}

\DeclareMathOperator{\Res}{Res}
\DeclareMathOperator{\Ind}{Ind}

\DeclareMathOperator{\vx}{vx}
\DeclareMathOperator{\Br}{Br}
\DeclareMathOperator{\rank}{rank}

\newcommand{\ol}{\overline}

\newcommand{\vphi}{\varphi}

\title{Lower defect groups and vertices of simple modules}
\author{Akihiko Hida and Masao Kiyota}

\address{Akhiko Hida, Faculty of Education, 
Saitama University,
Shimo-okubo 255, Sakura-ku, Saitama-city, Saitama, 338-8570, Japan}
\email{ahida@mail.saitama-u.ac.jp}

\address{Masao Kiyota,
College of Liberal Arts and Sciences, Tokyo Medical and Dental University, Kohnodai 2-8-30, Ichikawa, Chiba, 272-0827, Japan}
\email{kiyota.las@tmd.ac.jp}

\begin{document}
\maketitle

\begin{abstract}
We compare lower defect groups associated with 
$p$-regular classes 
and vertices of simple modules for a 
block of a finite group algebra. 
We show that lower defect groups are contained in vertices 
of simple modules 
after suitable reordering. 
Moreover, for a block of 
principal type, we show that 
a $p$-regular lower defect group which contains a vertex of simple module 
is a defect group of the block.  
\end{abstract}

\section{Introduction}\label{Introduction}

Let $G$ be a finite group. 
Let $\mcO$ be a complete discrete valuation 
ring of characteristic $0$ such that the 
residue field $k=\mcO/J(\mcO)$ is an algebraically 
closed field of characteristic $p>0$. 
Let $ZkG$ be the center of the group algebra $kG$. 
We call 
a primitive idempotent of $ZkG$ 
a block of $kG$. 
In this paper, we consider the relation between two 
important families of $p$-subgroups associated with a block $b$ 
of $kG$, lower defect groups of $b$ and vertices of simple 
$kGb$-modules. 
  
A conjugacy class of $G$ is called $p$-regular if 
the order of its element is not divisible by $p$. 
The $p$-regular conjugacy classes of $G$ are 
distributed to blocks of $kG$. 
Let $\Cl(G_{p'})$  be the set of all $p$-regular classes 
of $G$ and 
\[\Cl(G_{p'})=\bigcup_b \Cl_{p'}(b)\]
a decomposition of $\Cl(G_{p'})$ into blocks of $kG$. 
Let $\Cl_{p'}(b)=\{C_i\}_{1 \leq i \leq l(b)}$ and  
$Q_i$ a defect group of $C_i$, namely, 
$Q_i$ is a Sylow $p$-subgroup of $C_G(x_i)$ where 
$x_i \in C_i$. 
Then $Q_i$ ($1 \leq i \leq l(b))$ are called lower 
defect groups of $b$ associated with $p$-regular classes.

Let $\{S_i\}_{1 \leq i \leq l(b)}$ be a representatives of 
complete set of isomorphism classes of simple 
$kGb$-modules. It is known that this set and 
$\Cl_{p'}(b)$ have the same cardinality. 
We want to compare $Q_i$ and a vertex $\vx(S_j)$ of $S_j$. 
These are both subgroups of some defect group of $b$. 
For example, if a defect group $P$ of $b$ is abelian, 
then $\vx(S_i)=P$ for all $i$ by the theorem of Kn\"{o}rr \cite{K}. 
Hence in this case, $Q_i \leq_G \vx(S_j)$ for 
all $i,j$. 
Our first result shows that the inclusion 
$Q_i \leq_G \vx(S_i)$ holds in general for all $i$ if we 
renumber the indices suitably. 

\begin{prop}\label{permutation}
Let $b$ be a block of $kG$. 
Let $\{S_i\}_{1 \leq i \leq l(b)}$ be a set of representatives of isomorphism classes of simple $kGb$-modules. 
Let $\{Q_i\}_{1 \leq i \leq l(b)}$ be the lower defect groups of $b$ 
associated with $p$-regular classes. 
Then there exists a permutation $\sigma$ of $\{1,2,\dots,l(b)\}$ such 
that 
\[Q_i \leq_G \vx(S_{\sigma(i)})\]
for all $i$. In particular, 
\[\prod_{i=1}^{l(b)}|Q_i| \leq \prod_{i=1}^{l(b)}
|\vx (S_i)|.\]
\end{prop}

Next we consider when the equality holds in 
Proposition \ref{permutation}. 
We show that, under some assumption, if the equality holds 
then $Q_i=\vx(S_{\sigma(i)})$ is a defect group of $b$.  
A block $b$ is of principal type if 
the image of Brauer homomorphism with respect $Q$ is a 
block of $kC_G(Q)$ for every $p$-subgroup $Q$ contained in a 
defect group of $b$. 
For a $p$-subgroup $Q$ of $G$, 
we denote by $m_b^1(Q)$ the 
multiplicity of $Q$ as a lower defect group 
of $b$ associated with $p$-regular classes, that is, 
$m_b^1(Q)$ is a number of 
$i$ $(1 \leq i \leq l(b))$ such that $Q_i=_G Q$. 
The following theorem is our main result. 

\begin{thm}\label{principal}
Let $b$ be a block of principal type of $kG$ 
and $P$ a defect group of $b$. 
Let $S$ be a simple $kGb$-module and 
$R$ a vertex of $S$. 
If $R \leq Q <P$, then 
$m_b^1(Q)=0$.  
\end{thm}

By Proposition \ref{permutation} and Theorem \ref{principal}, 
we have the following Corollary. 

\begin{cor}\label{loweqvx}
Let $\{x_i\}_{1 \leq i \leq l}$ be a set 
of representatives of $p$-regular conjugacy classes 
of $G$. 
Let $Q_i$ be a Sylow $p$-subgroup of $C_G(x_i)$. 
Let $\{S_i\}_{1 \leq i \leq l}$ be a set of representatives of 
isomorphism classes of simple $kG$-modules. 
Then 
\[\prod_{i=1}^l |Q_i| \leq \prod_{i=1}^l|\vx(S_i)|\]
and the equality holds if and only if 
$G$ is $p$-nilpotent. 
\end{cor}

We prove these results in section \ref{Proof}.  
In section \ref{Blocks}, we consider blocks of $p$-solvable groups. 
We show a result which is slightly weaker than Theorem \ref{principal} 
without the assumption that $b$ is of principal type.   
Finally, in section \ref{Complexity}, we add some results on the 
complexity of modules. 

In this paper, all $kG$-modules are finite generated right 
modules. For an indecomposable $kG$-module $M$, 
we denote a vertex of $M$ by $\vx(M)$. 
We refer to \cite{F}, \cite{L1}, \cite{L2} and \cite{NT} for modular 
representations of finite groups.


\section{Lower defect groups of a block}\label{Lower}
Lower defect groups of a block are defined by Brauer \cite{Bra} 
and some related results can be found in \cite{B79}, \cite{I}, \cite{O} 
or \cite[V, Section 10]{F}, \cite[Chapter 5, Section 11]{NT}.
In this section  we quote some known results used in this paper. 
We consider lower defect groups associated with $p$-regular 
classes only. 

Let $G_{p'}$ be the set of all $p$-regular elements 
of $G$. We denote by $kG_{p'}$ the $k$ subspace of $kG$ spanned  by $G_{p'}$. 
If $C$ is a conjugacy class of $G$, we set 
$\hat{C}=\sum_{x \in C}x \in kG$. Then 
$\{\hat{C} \ | \ C \in \Cl(G)\}$ is a basis 
of the center $ZkG$ of $kG$.
Let $\Cl(G_{p'})$  be the set of all $p$-regular 
conjugacy classes of $G$. 
Then $\{\hat{C} \ | \ C \in \Cl(G_{p'})\}$ is a basis of 
$ZkG_{p'}=ZkG \cap kG_{p'}$. 
On the other hand, 
\[ZkG_{p'}=\bigoplus_b ZkG_{p'}b=\bigoplus_b ZkG_p'\cap ZkGb\]
where $b$ ranges over all blocks of $kG$ and  
there exists a disjoint decomposition  
\[\Cl(G_{p'})=\bigcup_b \Cl_{p'}(b)\]
such that 
$\{\hat{C} b \ | \ C \in \Cl_{p'}(b)\}$ is a basis of 
$ZkG_{p'}b$ for any block $b$. 
 
Let $\Cl_{p'}(b)=\{C_i\}_{1 \leq i \leq l(b)}$ and  
$Q_i$ a defect group of $C_i$, that is, 
$Q_i$ is a Sylow $p$-subgroup of $C_G(x_i)$ where 
$x_i \in C_i$. 
We denote by $m_b^1(Q)$ the 
multiplicity of $Q$ as a lower defect group 
of $b$ associated with $p$-regular classes, that is, 
$m_b^1(Q)$ is a number of 
$i$ $(1 \leq i \leq l(b))$ such that $Q =_G Q_i$. 
Here, for subgroups $H,K$ of $G$, 
we write $H=_GK$ if $H$ and $K$ are $G$ conjugate.  

Let $M$ be a $kG$-module and $H \subset K$ be subgroups of $G$. 
Let $M^H$ be the set of fixed points of $H$ in $M$. 
We denote by $\Tr_H^K$ the trace map from $M^H$ to 
$M^K$ defined by 
\[\Tr_H^K(m)=\sum_{g \in H \backslash K}mg.\]
We set $M_H^K=\Tr_H^K(M^H)$. 
The group $G$ acts on $kG$ and $kG_{p'}$ by 
conjugation and we can define the trace maps  
\[\Tr_H^K: (kG)^H \lorarr (kG)^K\]
and 
\[\Tr_H^K: (kG_{p'})^H \lorarr (kG_{p'})^K. \]

The proof of the following basic properties of 
lower defect groups are found in \cite[V, Section 10]{F} and 
\cite[Chapter 5, Section 11]{NT}. 
 
\begin{prop} 
Let $b$ be a block of $kG$ and $P$ a defect group 
of $b$. Then the following holds. 
\\
(1) If $Q$ is a $p$-subgroup of $G$, then $m_b^1(Q)=\dim (kG_{p'})_Q^G/\sum_{R<Q} (kG_{p'})_R^G$.
\\
(2) If $m_b^1(Q)>0$ then $Q$ is conjugate to a subgroup 
of $P$ and $m_b^1(P)=1$. 
\\
(3) Let $\Cl_{p'}(b)=\{C_i\}_{1 \leq i \leq l(b)}$ and $Q_i$ a defect 
group of $C_i$. Then 
$l(b)=|\Cl_{p'}(b)|$ is the number of isomorphism classes of 
simple $kGb$-modules and $\{|Q_i| \ | \ 1\leq  i \leq l(b)\}$ is the 
set of elementary divisors of the Cartan matrix of $kGb$.  
\\
(4) Let $m \geq 0$. Then 
$\sum_Q m_b^1(Q)$ where $Q$ ranges over the 
set of representatives of conjugacy classes of subgroups of $G$ of order $p^m$ 
is the multiplicity of $p^m$ as an elementary 
divisor of the Cartan matrix of $kGb$.
\end{prop}

We need the following result in Lemma \ref{normal} and 
Theorem \ref{allcentric}.

\begin{prop}[{\cite[Proposition (II) 1.3]{B79}}]
\label{ldBroue}
If $Q$ is a $p$-subgroup of $G$ then
\[m_b^1(Q)=
\dim (kC_G(Q)_{p'})_Q^{N_G(Q)}\Br_Q(b)\]
where $\Br_Q: ZkG \lorarr ZkN_G(Q)$ is the 
Brauer homomorphism. 
\end{prop}


\section{Proof of main results}\label{Proof}

The following Proposition is proved in 
{\cite[Theorem 10, Corollary]{G}}. 
It is obtained from 
{\cite[p.243, Corollaire]{B76}} or   
{\cite[IV, Theorem 2.3]{F}} also.

 \begin{prop}\label{Green}
Let $M$ be an indecomposable  $kG$-module. 
\\
(1) Let $\vphi$ be the Brauer character 
corresponding to $M$. 
Let $x \in G_{p'}$ and $Q \in \Syl_p(C_G(x))$. If 
\[\vphi(x) \not\in p\mcO\]
then $\Res^G_Q M$ has an indecomposable direct summand $N$ 
such that $\dim N \not\equiv 0 \bmod p$. 
\\
(2) If $Q$ is a $p$-subgroup of $G$ and $\Res^G_Q M$ has an indecomposable direct summand $N$ 
such that $\dim N \not\equiv 0 \bmod p$, then $Q \leq_G \vx(M)$.
\end{prop}

Proposition \ref{permutation} is immediate from 
this Proposition. 
\vspace{.5cm}\\
(Proof of Proposition \ref{permutation})

Let $\vphi_i$ be the Brauer character of $S_i$. 
Let $\Cl_{p'}(b)=\{C_i\}_{1 \leq i \leq l(b)}$ be the $p$-regular 
conjugacy classes distributed into $b$ in the 
block decomposition of $\Cl(G_{p'})$. 
Fix $x_i \in C_i$ for each $i$. We may assume that 
$Q_i$ is a Sylow $p$-subgroup of $C_G(x_i)$ where $x_i \in C_i$.    
Then the determinant of the matrix 
$(\vphi_i(x_j))_{1 \leq i,j \leq l(b)}$ is not contained in 
$J(\mcO)$ by \cite[Chapter 5, Theorem 11.6]{NT} 
and there exists a permutation $\sigma$ 
such that 
\[\prod_{i} \vphi_{\sigma(i)}(x_i) \not\in J(\mcO).\]
Since $p\mcO \subset J(\mcO)$, the result follows from Proposition \ref{Green}. 
\vspace{.5cm}

Next we prove Theorem \ref{principal}. 
First we study the case that $Q$ is a normal subgroup of $G$.

\begin{lem}\label{trace0}
Let $G \triangleright H$. 
Let $b$ be a block of $kH$ of defect $0$. 
Let $T=T(b)$ be the inertial group of $b$ in $G$. 
Assume that $|T:H| \equiv 0 \bmod p$. 
Then $b(kH)_1^G=0$. 
\end{lem} 

\begin{proof}
By the Mackey decomposition, we have 
\[(kH)_1^G\subset 
\sum_{t \in G/T}((kH)^t)_1^T=(kH)_1^T.\]
Since $b$ is $T$-invariant, we have  
\[b(kH)_1^T=b \Tr_1^T(kH)=
\Tr_1^T(bkH)=\Tr_H^T(\Tr_1^H(b kH)).\]
Then 
\[\Tr_1^H(bkH)=Z(bkH)=kb\]
since $b$ is a block of defect $0$, and  
\[
\Tr_H^T(\Tr_1^H(b kH))=
\Tr_H^T(kb)=0\]
since $|T:H| \equiv 0 \bmod p$.
\end{proof}

\begin{lem}\label{QCGQ/Q}
Let $Q$ be a normal $p$-subgroup of $G$ and 
$\ol{G}=G/Q$. 
Let $\mu : kG \lorarr k\ol{G}$ be the surjective 
algebra homomorphism induced by the natural 
surjective homomorphism $G \lorarr \ol{G}$. 
Let $b$ be a block of $kG$ and 
$\ol{b}=\mu(b)$. 
Then $\mu$ induces an isomorphism 
\[b\cdot (kC_G(Q)_{p'})_Q^G 
\cong 
\ol{b} \cdot (k(QC_G(Q)/Q)_{p'})_{\ol{1}}^{\ol{G}}\] 
\end{lem}

\begin{proof}
The natural surjective homomorphism $G \lorarr \ol{G}$ 
induces a bijection 
\[C_G(Q)_{p'} \lorarr \ol{(QC_G(Q))}_{p'}=(QC_G(Q)/Q)_{p'}.\]
Hence $\mu$ induces an isomorphism 
\[\mu_0 : 
kC_G(Q)_{p'} \simeq k(\ol{QC_G(Q)})_{p'}.\]
On the other hand, since
\[\mu (b\cdot (kC_G(Q)_{p'})_Q^G)=
\ol{b} \cdot (k(\ol{QC_G(Q)})_{p'})_{\ol{1}}^{\ol{G}}\]
and 
\[b\cdot (kC_G(Q)_{p'})_Q^G\subset kC_G(Q)_{p'}, \ 
\ol{b} \cdot (k(\ol{QC_G(Q)})_{p'})_{\ol{1}}^{\ol{G}}
\subset k(\ol{QC_G(Q)})_{p'}, \]
it follows that the restriction of $\mu_0$ induces 
the desired isomorphism. 
\end{proof}

\begin{lem}\label{normal}
Let $Q$ be a normal $p$-subgroup of $G$ and $b$ a block of $kG$. 
Let $P$ be a defect group of $b$. 
Assume that  $Q <P$ and $C_P(Q)=Z(Q)$. 
Then $b\cdot(kC_G(Q)_{p'})_Q^G=0$, in 
particular, $m_b^1(Q)=0$. 
\end{lem}

\begin{proof} 
Let $H=QC_G(Q)$. The block $b$ is a central idempotent of $kH$. 
Let $b=\sum_i b_i$ be the block decomposition of $b$ 
in $kH$. 
Let $T_i=T(b_i)$ be the inertial group of $b_i$ in $G$. 
Then there exists a defect group $P_i$ of $b$ such 
that $P_i \leq T_i$ and $P_i \cap H$ is a 
defect group of $b_i$. 
Since $P_i$ is conjugate to $P$ in $G$ and 
$C_P(Q)=Z(Q)$, 
\[P_i \cap H=P \cap H=Q\]
and it follows that $Q$ is a defect group of $b_i$. 
Let $\mu : kG \lorarr k\ol{G}=k(G/Q)$ be the 
natural surjective algebra homomorphism. 
Then $\mu(b_i)=\ol{b_i}$ is a block of 
$k\ol{H}$ of defect $0$ for all $i$. 
Since $Q<P$, we have 
$H \ne HP_i$ and $|T_i:H| \equiv 0 \bmod p$. 
The inertial group of $\ol{b_i}$ in $\ol{G}$ is 
$\ol{T_i}=T_i/Q$ and $|\ol{T_i}: \ol{H}|\equiv 0 \bmod p$. 
Hence 
\[\ol{b_i}(k\ol{H})_{\ol{1}}^{\ol{G}}=0\]
by Lemma \ref{trace0} and 
\[b(k\ol{H})_{\ol{1}}^{\ol{G}}=
(\sum_i \ol{b_i}) (k\ol{H})_{\ol{1}}^{\ol{G}}=0.\]
Then, by Lemma \ref{QCGQ/Q},  
\[b(kC_G(Q)_{p'})_Q^G \simeq 
\ol{b}(k\ol{H}_{p'})_{\ol{1}}^{\ol{G}}
\subset 
\ol{b}(k\ol{H})_{\ol{1}}^{\ol{G}}=0.\]
Moreover, since $Q$ is a normal subgroup of $G$, 
$\Br_Q(b)=b$ and the result follows from Proposition 
\ref{ldBroue}.  
\end{proof}

In the following, we consider a $p$-subgroup $Q$ 
such that $C_P(Q)=Z(Q)$ for any defect group $P$  
of $b$ containing $Q$. 
  
\begin{thm}\label{allcentric}
Let $b$ be a block of $kG$ and $Q$ a $p$-subgroup 
of $G$. 
Assume that $Q$ is a proper subgroup of a 
defect group of $b$. If $C_P(Q)=Z(Q)$ 
for any defect group $P$ of $b$ containing $Q$, 
then $m_b^1(Q)=0$. 
\end{thm}

\begin{proof}
Let $N=N_G(Q)$. 
Let 
\[\Br_Q: ZkG \lorarr ZkN\]
be the Brauer homomorphism and 
\[\Br_Q(b)=\sum_i b_i\]
a decomposition into block idempotents of $kN$. 
Then $b_i^G=b$ and $Q<D_i$ by Brauer's First 
Main Theorem where $D_i$ is a defect group 
of $b_i$. 
There exists a defect group $P_i$ of $b$ such 
that $D_i \leq P_i$. 
Then 
\[C_{D_i}(Q) \leq C_{P_i}(Q)=Z(Q)\]
by the assumption. Hence 
\[b_i \cdot (kC_G(Q)_{p'})_Q^N =0\]
by Lemma \ref{normal} and 
\[(kC_G(Q)_{p'})_Q^N \Br_Q(b)
=(kC_G(Q)_{p'})_Q^N (\sum_ib_i)=0.\]
Hence the result follows by Proposition \ref{ldBroue}.
\end{proof}

\begin{lem}\label{Fcentric} 
Let $b$ be a block of $kG$ and $(P,e)$ a maximal 
$(G,b)$-Brauer pair. Let $\mcF=\mcF_{(P,e)}(G,b)$ 
be the fusion system of $b$ with respect to $(P,e)$. 
Suppose that $\mcF=\mcF_P(G)$. 
Let $Q$ be a proper subgroup of $P$. If 
$Q$ is $\mcF$-centric, then $m_b^1(Q)=0$. 

\end{lem}

\begin{proof}
Let $P_1$ is a defect group of $b$ such that 
$Q <P_1$. Then 
there exists $g \in G$ such that 
$P^g=P_1$. Let $Q_1={^gQ}$.
Then 
\[\vphi: Q \lorarr Q_1 (\leq P), \ 
\vphi(u)=gug^{-1}\]
is an isomorphism in $\mcF=\mcF_P(G)$. 
Hence $C_P(Q_1) = Z(Q_1)$ since $Q$ is 
$\mcF$-centric and it follows that 
$C_{P_1}(Q)=Z(Q)$. Hence 
$b$ satisfies the assumption of Theorem \ref{allcentric}. 

\end{proof}

Now we prove Theorem \ref{principal} and 
Corollary \ref{loweqvx}.  
A block $b$ of $kG$ is of principal type if 
$\Br_Q(b)$ is a block of $kC_G(Q)$ for 
every $p$-subgroup $Q$ contained in a 
defect group of $b$ (\cite[Definition 6.3.13]{L2}). 
\vspace{.5cm}\\
(Proof of Theorem \ref{principal})

There exists a $(G,b)$-Brauer pair $(R,f)$ such that 
$Z(R)$ is a defect group of the block $f$ of $kC_G(R)$ by 
the theorem of Kn\"{o}rr (\cite[3.6 Corollary]{K}, 
\cite[Corollary 10.3.2]{L2}).  
Let $(P,e)$ be a maximal $(G,b)$-Brauer pair such that 
$(R,f) \leq (P,e)$. Let  $\mcF=\mcF_{(P,e)}(G,b)$ be the 
fusion system of $b$ with respect to $(P,e)$. 
Then 
$R$ is an $\mcF$-centric subgroup of $P$ 
by \cite[Proposition 8.5.3]{L2}. 
Moreover, since $b$ is a block of principal type, 
we have 
$\mcF=\mcF_P(G)$ by \cite[Proposition 8.5.5]{L2}. 
Since 
$R \leq Q \leq P$, 
$Q$ is an $\mcF$-centric subgroup by 
\cite[Proposition 8.2.4]{L2}. 
Hence the results follows from Lemma \ref{Fcentric}.

\begin{rem}
{\rm 
In Theorem \ref{principal}, $S$ does not necessarily need to be simple. 
Let $b$ be a block of $\mcO G$ with defect group $P$ and $S$ an indecomposable 
$\mcO G b$-module. If $\underline{End}_{\mcO Gb}(S) \cong 
\mcO/J(\mcO)^m$ for some $m>0$, then there is a 
$(G,b)$-Brauer pair $(R,f)$ such that $Z(R)$ is a defect  group 
of the block $f$ of $kC_G(R)$ as in the proof of Theorem \ref{principal} 
by \cite[Corollary 10.3.2]{L2}. 
It follows that if $b$ is of principal type and 
$R \leq Q<P$, then $m_b^1(Q)=0$. 
}
\end{rem}

\vspace{.3cm}

\noindent
(Proof of Corollary \ref{loweqvx})

The inequality holds by Proposition \ref{permutation}. 
If $G$ is $p$-nilpotent, 
then $kGb$ has a unique simple module (up to 
isomorphism) for every block $b$ of $kG$. 
The vertex of the simple $kGb$-module is 
a defect group of $b$ and that is a 
$p$-regular lower defect group of $b$. 
Hence the equality holds. 
On the other hand, assume that the equality holds. 
If $P$ is a defect group of a block $b$ of 
$kG$, then $m_b^1(P)=1$. Hence 
the principal block $b_0(kG)$ of $kG$ has 
a unique simple module  (up to isomorphism) by 
Proposition \ref{permutation} and  
Theorem \ref{principal} since $b_0(kG)$ is a block 
of principal type by Brauer's Third Main Theorem 
(\cite[Theorem 6.3.14]{L2}). 
Hence $G$ is $p$-nilpotent. 


\section{Blocks of $p$-solvable groups}\label{Blocks}

Let $N$ be a normal subgroup of $G$. 
Let $b$ be a block of $kG$. Let  
$c$ be the block of $kN$ such that 
$bc \ne 0$ and 
$H$ the inertial 
group of $c$ in $G$. Then there exists a block 
$d$ of $kH$ such that 
$db=d$ and $\Tr_H^G(d)=b$. 
The $(kGb, kHd)$-bimodule $bkGd=kGd$ induces 
a Morita equivalence between $kGb$ and $kHd$. 
In particular, $l(b)=l(d)$. 
The trace map induces an 
algebra isomorphism 
\[\Tr_H^G: Z(kHd) \lorarr Z(kGb).\]
Moreover if $P$ is a 
defect group of $d$ then $P$ is a defect group 
of $b$ (cf.\cite[Theorem 6.8.3]{L2}). 
With these notations, we have the following 
two lemmas. 

\begin{lem}\label{clifford-trace}
If $Q$ is a subgroup of $H$, then 
$\Tr_H^G$ induces an isomorphism 
\[\Tr_H^G: 
\sum_{t \in G}(kH_{p'})_{Q^t \cap H}^H d 
\simeq (kG_{p'})_Q^Gb. \]
\end{lem}

\begin{proof}
$\Tr_H^G$ induces an injective $k$-linear map   
\[\Tr_H^G: (kH_{p'})_R^Hd \lorarr (kG_{p'})_R^Gb\]
for any subgroup $R$ of $H$. 
In particular, 
\[\Tr_H^G((kH_{p'})_{Q^t\cap H}^H d)
\subset (kG_{p'})_{Q^t}^Gb=
(kG_{p'})_Q^Gb\]
for $t \in G$. 
On the other hand, 
\[(kG_{p'})_Q^G \subset \sum_{t \in G} 
(kG_{p'})_{Q^t \cap H}^H\]
and 
\begin{eqnarray*}
(kG_{p'})_Q^Gd \subset 
\sum_{t \in G} (kG_{p'})_{Q^t \cap H}^Hd 
&=&
\sum_{t \in G} 
((kH_{p'})_{Q^t \cap H}^H +
k(G_{p'}-H_{p'})_{Q^t \cap H}^H)d \\ 
&=&
\sum_{t \in G} 
(kH_{p'})_{Q^t \cap H}^Hd
\end{eqnarray*}
since $dgd=0$ for any $g \in G-H$.   
Hence 
\[(kG_{p'})_Q^Gb=\Tr_H^G((kG_{p'})_Q^Gd)
\subset \Tr_H^G(\sum_{t \in G} 
(kH_{p'})_{Q^t \cap H}^Hd).\]
\end{proof}

\begin{lem}\label{clifford-multiplicity}
Let $U$ be a $p$-subgroup of $H$. 
Let 
\[\mcU=\{W \leq H~|~W =_G U\}\]
and let $\{U_j\}$ be a set of representatives 
of $H$-conjugacy classes of $\mcU$. 
Then 
\[\sum_{j} m_d^1(U_j)=m_b^1(U).\]
\end{lem}

\begin{proof}
Let $\{R_{ij}\}_{1 \leq i \leq m, 1 \leq j \leq r(i)}$ 
be a set of representatives of $H$-conjugacy classes 
of subgroups of $H$ of order $|U|$ such that 
\[R_{ij}=_G R_{i'j'} \iff i =i'.\]
We may assume $U=R_{11}$ and 
$\{U_j\}=\{R_{1j}\}$. We set $R_i=R_{i1}$. For each $i$, $\Tr_H^G$ induces a $k$-linear map
\[\Phi_i: \bigoplus_{1\leq j \leq r(i)}
\left(
(kH_{p'})_{R_{ij}}^Hd \Big/ \sum_{R<R_{ij}} (kH_{p'})_R^Hd
\right) \lorarr
(kG_{p'})_{R_i}^Gb \Big/ \sum_{R<R_i}(kG_{p'})_R^Gb.\]
We claim that $\Phi_i$ is an isomorphism.   
For $t \in G$, if $R_i^t \leq H$ then 
$R_i^t =_H R_{ij}$ for some $j$. On the other hand, 
if $R_i^t \cap H<R_i^t$ then 
\[\Tr_H^G((kH_{p'})_{R_i^t\cap H}^H d)
\subset 
(kG_{p'})_{R_i^t \cap H}^Gb 
=
(kG_{p'})_{R_i \cap {}^tH}^Gb 
\subset 
\sum_{R<R_i}(kG_{p'})_R^Gb.\]
Hence we have 
\[(kG_{p'})_{R_i}^Gb=
\sum_{t\in G}
\Tr_H^G((kH_{p'})_{R_i^t \cap H}^H d)
\subset 
\Tr_H^G(\sum_j(kH_{p'})_{R_{ij}}^Hd)+
\sum_{R<R_i}(kG_{p'})_R^Gb\]
by Lemma \ref{clifford-trace} and $\Phi_i$ is surjective. 
Hence 
\[\sum_{1\leq j \leq r(i)}m_d^1(R_{ij}) 
\geq m_b^1(R_i)\]
for each $i$ and 
\[\sum_{i,j}m_d^1(R_{ij}) 
\geq 
\sum_i m_b^1(R_i).\]
But $\sum_{i,j}m_d^1(R_{ij})$ is the multiplicity 
of $|U|$ as an elementary divisor of the Cartan matrix $C_d$ of $kHd$ 
and $\sum_i m_b^1(R_i)$ is that of the Cartan matrix $C_b$ 
of $kGb$. 
Since $kGb$ and $kHd$ are Morita equivalent, 
$C_d=C_b$ and it follows that 
\[\sum_{i,j}m_d^1(R_{ij}) 
= 
\sum_i m_b^1(R_i).\]
Hence 
\[\sum_{1\leq j \leq r(i)}m_d^1(R_{ij}) 
= m_b^1(R_i) \]
and $\Phi_i$ is an isomorphism for each $i$. 
In particular, for $i=1$, we have 
\[\sum_{j} m_d^1(U_j)=m_b^1(U).\]
\end{proof}

The following theorem is the main result of this section. 
If the block $b$ is of principal type, then 
this is a consequence of Theorem \ref{principal}. 
 
\begin{thm}\label{psolvable}
Let $G$ be a $p$-solvable group. 
Let $b$ be a block of $kG$. 
Let $\Cl_{p'}(b)=\{C_i\}_{1 \leq i \leq l(b)}$ 
and $Q_i$ a defect group of $C_i$. 
Let $\{S_i\}_{1 \leq i \leq l(b)}$ be a set of representatives of isomorphism classes of simple $kGb$-modules. 
Then there exists a permutation $\sigma$ of 
$\{1,\dots, l(b)\}$ such that 
\[Q_i \leq_G \vx(S_{\sigma(i)})\]
for all $i$ and 
\[Q_i <_G \vx(S_{\sigma(i)})\]
unless $Q_i$ is a defect group of $b$.  
\end{thm}

\begin{proof}
Let $c$ be a block of $kO_{p'}(G)$  
such that $bc \ne 0$ and $H$ the inertial 
group of $c$ in $G$. Then there exists a block 
$d$ of $kH$ such that 
$db=d$ and $\Tr_H^G(d)=b$. 
Moreover $(kGb, kHd)$-bimodule $bkGd=kGd$ induces 
a Morita equivalence between $kGb$ and $kHd$. 
In particular, $l(b)=l(d)$. Moreover if $P$ is a 
defect group of $d$ then $P$ is a defect group 
of $b$. 

If $G=H$, then the result holds by Proposition 
\ref{permutation} and Theorem \ref{principal} 
since $b=d$ is a block of principal type by 
\cite[Lemma 10.6.5]{L2}. 

Suppose that $G>H$. Let $\Cl_{p'}(b)=\{C_i\}$ and 
$\Cl_{p'}(d)=\{\tilde{C}_i\}$. 
Let $Q_i$ (resp. $\tilde{Q}_i$) be a defect group of $C_i$ (resp. $\tilde{C}_i$). 
If $Q$ is a $p$-subgroup of $P$, 
\[
|\{1 \leq j \leq l(d)~|~ \tilde{Q}_j =_G Q\}|
=m_b^1(Q)=
|\{1 \leq i \leq l(b)~|~Q_i =_G Q\}|
\]
by Lemma 
\ref{clifford-multiplicity}.
Hence we may assume $Q_i=\tilde{Q_i}$ for every 
$1 \leq i \leq l(b)$. 
Let $\tilde{S_i}=S_i \otimes_{kGb}kGd$ 
be the simple $kHd$-module corresponding to $S_i$. 
Then $\vx(\tilde{S_i})=_G\vx(S_i)$. 
By induction there exists a permutation 
$\sigma$ of $\{1, \dots, l(d)\}$ such that 
\[Q_i \leq_H \vx(\tilde{S}_{\sigma(i)})\]
for all $i$ and 
\[Q_i <_H \vx(\tilde{S}_{\sigma(i)})\]
if $Q_i<_H P$.   
Since $P$ is a defect group of $b$, it follows that 
\[Q_i \leq_G \vx(S_{\sigma(i)})\]
for all $i$ and 
\[Q_i <_G \vx(S_{\sigma(i)})\]
unless $Q_i$ is a defect group of $b$.  
\end{proof}

The following corollary is a block version of 
Corollary \ref{loweqvx} for $p$-solvable groups. 
 
\begin{cor}
Let $G$ be a $p$-solvable group. 
Let $b$ be a block of $kG$. 
Let $\{S_i\}_{1 \leq i \leq l(b)}$ be a set of representatives of isomorphism classes of simple $kGb$-modules. 
Then
\[\det C_b \leq \prod_{i=1}^{l(b)}|\vx(S_i)|\]
where $C_b$ is the Cartan matrix of $kGb$ 
and the equality holds 
if and only if $l(b)=1$. 
\end{cor}


\section{Complexity of modules}\label{Complexity}

Let $Q$ be a $p$-subgroup of $G$ 
and $M$ an indecomposable 
$kG$-module. 
Suppose that $\Res^G_Q M$ has a 
direct summand $N$ such that $\dim N \not\equiv 0 \bmod p$.
We set $|\vx(M)|=p^{v(M)}$ and $|Q|=p^a$. 
Then by Proposition \ref{Green}, $Q \leq_G \vx(M)$ and 
in particular $a \leq v(M)$. If the dimension of a source of $M$ is 
divisible by $p$, then proper inequality $a<v(M)$ holds. We 
consider another information on this inequality related to the 
complexity of $M$. 
For the complexity of a module, we refer to 
\cite{AE} and \cite[Section 5]{Be}.

Let $c(M)$ be the complexity of $M$ and 
$r(M)$ the $p$-rank of $\vx(M)$. 
Since $M$ is a direct summand of $\Ind_{\vx(M)}^G \Res^G_{\vx(M)}M$  
and $c(\Res^G_{\vx(M)} M) \leq r(M)$, we have 
\[c(M) \leq r(M) \leq v(M).\]

\begin{prop}\label{lvcomplexity}
Let $Q$ be a $p$-subgroup of $G$ 
and $M$ an indecomposable 
$kG$-module. 
Suppose that $\Res^G_Q M$ has a 
direct summand $N$ such that $\dim N \not\equiv 0 \bmod p$. If $|Q|=p^a$ then 
\[a \leq v(M)+c(M)-r(M).\]
\end{prop}

\begin{proof}
Since $N$ is a direct summand of $\Res^G_Q M$ and 
$\dim N \not\equiv 0 \bmod p$, 
\[\rank(Q)= c(N)\leq c(\res^G_Q M) \leq c(M)\]
and  
\[\rank(Q) +v(M)-r(M)\leq v(M)+c(M)-r(M)\]
where  $\rank(Q)$ is the $p$-rank of $Q$.  
We may assume that $Q \leq \vx(M)$ by Proposition 
\ref{Green}(2). 
Let $E$ be an elementary abelian $p$-subgroup 
of $\vx(M)$ of maximal rank and 
$F$ an elementary abelian $p$-subgroup 
of $Q$ of maximal rank.
The class of elementary abelian $p$-groups 
satisfies the condition in Lemma \ref{maxgrouplem} below. 
Hence $|Q|/|F| \leq |\vx(M)|/|E|$ by Lemma \ref{maxgrouplem} and 
we have  
\[a \leq \rank(Q) +v(M) -r(M).\]
\end{proof}

\begin{lem}\label{maxgrouplem}
Let $\mcX$ be a class of finite $p$-groups which 
satisfies the following property:
\[P \in \mcX, \ P \triangleright R \Rarr R \in \mcX. \]
Let $P$ be a $p$-group and $Q$ be a subgroup 
of $P$. 
Suppose $E \leq P$, $F \leq Q$ and $E,F \in \mcX$. 
If $|L| \leq |F|$ for any subgroup $L \leq Q$ such 
that $L \in \mcX$, 
then 
\[|Q:F| \leq |P:E|.\]
\end{lem}

\begin{proof}
We proceed by induction on $|P:Q|$. 
If $P=Q$ then $|E| \leq |F|$ by the assumption. 
Hence $|Q:F| \leq |P:E|$. 
Assume that $P >Q$ and let 
$R$ be a maximal subgroup of $P$ such that 
$Q \leq R <P$. 
Then $R \cap E \in \mcX$ since 
$P \triangleright R$ and $E \triangleright R \cap E$.
If $R \not\geq E$, then $P=RE$ and $P/R=RE/R \cong E/R\cap E$. It follows that 
$|P|/|E|=|R|/|R\cap E|$ and  
$|R|/|R \cap E| \geq |Q|/|F|$ 
by induction.
If $R \geq E$, then $|P|/|E| >|R|/|E|$ and 
it follows that $|R|/|E| \geq |Q|/|F|$ by induction.
\end{proof}

\end{document}